\documentclass[final,12pt]{alt2025} %

\title
[Aggregate expert error]
{
Sharp bounds on aggregate expert error
}
\usepackage{times}
\hypersetup{
    colorlinks = true,
    urlcolor = blue,
    linkcolor = blue,
    citecolor = blue
}

\altauthor{%
 \Name{Aryeh Kontorovich} \Email{karyeh@cs.bgu.ac.il}\\
 \addr Computer Science Department, Ben-Gurion University of the Negev
Beer Sheva, Israel
\AND
 \Name{Ariel Avital} \Email{avitalq@post.bgu.ac.il}\\
 \addr Computer Science Department, Ben-Gurion University of the Negev
Beer Sheva, Israel
}

\newcommand{\qed}{$\blacksquare$}

\newcommand{\pred}[1]{\boldsymbol{1}[#1]}

\newcommand{\R}{\mathbb{R}}

\newcommand{\N}{\mathbb{N}}

\DeclareMathOperator{\sgn}{sign}

\newcommand{\abs}[1]{\left| #1 \right|}
\newcommand{\paren}[1]{\left( #1 \right)}
\newcommand{\sqprn}[1]{\left[ #1 \right]}
\newcommand{\nrm}[1]{\left\Vert #1 \right\Vert}

\newcommand{\vertiii}[1]{{\left\vert\kern-0.25ex\left\vert\kern-0.25ex\left\vert #1 
    \right\vert\kern-0.25ex\right\vert\kern-0.25ex\right\vert}}
\newcommand{\tv}[1]{\nrm{#1}_{\textup{\tiny\textsf{TV}}}}
\newcommand{\opt}{^{\textup{\tiny\textsf{OPT}}}}

\newcommand{\mexp}{\mathbb{E}}

\newcommand{\E}{\mathop{\mexp}}
\renewcommand{\P}{\mathbb{P}}

\newcommand{\eps}{\varepsilon}

\usepackage{xcolor}

\newcommand{\beq}{\begin{eqnarray*}}
\newcommand{\eeq}{\end{eqnarray*}}
\newcommand{\beqn}{\begin{eqnarray}}
\newcommand{\eeqn}{\end{eqnarray}}

\usepackage{ifmtarg}
\usepackage{xifthen}%
\newcommand{\ent}[1][]{%
\ifthenelse{\isempty{#1}}{%
\mathrm{H}
}{
\mathrm{H}^{(#1)}
}}

\newcommand{\loch}[1][]{%
\ifthenelse{\isempty{#1}}{%
\mathrm{h}
}{
\mathrm{h}^{(#1)}
}}

\newcommand{\mathe}{\mathrm{e}}

\newcommand{\hide}[1]{}

\renewcommand{\set}[1]{\left\{ #1 \right\}}

\newcommand{\Ber}{\operatorname{Ber}}

\newcommand{\track}[1]{
\colorlet{dcsaved}{.}
\color{red}
#1
\color{dcsaved}
}
\renewcommand{\track}[1]{#1}

\makeatletter
\@addtoreset{theorem}{section}
\makeatother

\begin{document}

\maketitle

\begin{abstract}
We revisit the classic problem of aggregating
binary advice from conditionally independent experts,
also known as the Naive Bayes setting.
Our quantity of interest is 
the
error probability
of the optimal decision rule.
In the 
case
of
symmetric 
errors
(sensitivity = specificity),
reasonably tight bounds on the optimal error probability
are known. In the general asymmetric case, we are not
aware of any nontrivial estimates on this quantity.
Our contribution consists of sharp
upper and lower bounds on the
optimal error probability in the general case,
which recover and sharpen the best known results
in the symmetric special case.
Additionally, our bounds are apparently
the first to take the bias into account.
Since this 
turns out to be closely connected to bounding
the total variation
distance between two product distributions,
our results also have bearing on this important and challenging problem.
\end{abstract}


\section{Introduction}

Consider the following 
decision-theoretic setting.
A 
parameter $\theta\in(0,1)$
is fixed
and
a random bit $Y\in\set{0,1}$
is drawn 
according to Bernoulli with bias $\theta$:
that is, $\theta=\P(Y=1)=1-\P(Y=0)$.
Conditional on $Y$, the 
$\set{0,1}$-valued
variables
$X_1,X_2,\ldots,X_n$
are drawn 
independently
according to
\beqn
\P(X_i=1|Y=1) &=& \psi_i, 
\label{eq:psi}
\\
\label{eq:eta}
\P(X_i=0|Y=0) &=& \eta_i
\eeqn
for some collection of parameters $\psi,\eta\in(0,1)^n$.
The $\psi_i$ and $\eta_i$
are classically known as {\em sensitivity}
and {\em specificity}, respectively.
An agent who knows the values of $\track{\theta},\psi,\eta$ 
gets to
observe $X=(X_1,\ldots,X_n)$
and wishes to infer the most likely $Y$ conditional on $X$.
A decision
rule
$f:\set{0,1}^n\to\set{0,1}$
that minimizes
the 
{\em error probability} $\P(f(X)\neq Y)$
may be found
in
\citet[Eqs. (11), (12)]{boaz-ranking2014}:\footnote{
The bound therein was stated for the case $\theta=1/2$
---
i.e., without the $\log\frac{\theta}{1-\theta}$
term. We rederive the full expression for completeness in
Section \ref{sec:pf-errp}.
}
\beqn
\label{eq:fopt}
f\opt
:X \mapsto
\sgn\paren{
\log\frac{\theta}{1-\theta}
+
\sum_{i=1}^n 
(2
X_i
-1)
\log\alpha_i+\log\beta_i
},
\eeqn
where
\beq
\alpha_i=\frac{\psi_i\eta_i}{(1-\psi_i)(1-\eta_i)},
\quad \beta_i=\frac{\psi_i(1-\psi_i)}{\eta_i(1-\eta_i)}
,
\eeq
and $\sgn$,
along with the rest of our 
notation,
is defined below.

The 
main quantity of interest 
in this note
is the
optimal error probability $\P(f\opt(X)\neq Y)$.
We obtain improved, sharp bounds on this quantity
in the symmetric ($\psi=\eta$) and asymmetric cases.
It will turn out that
estimating
$\P(f\opt(X)\neq Y)$
is closely related to computing
the total variation
distance between two product distributions
---
and thus
our results also have bearing on this important and 
computationally
challenging 
problem.

\paragraph{Motivation.}
The Neyman-Pearson Lemma (see \eqref{eq:NP})
lies at the heart of decision theory and hypothesis
testing, as it provides an optimal 
risk-minimizing
strategy. Our results
continue a line of work that analyzes the {\em performance} of this optimal strategy. 
Plans for future work include finite-sample
guarantees based on the spectral estimates of
\citet{boaz-ranking2014}.

\paragraph{Definitions.}
The
{\em balanced accuracy} 
is defined as
$\pi_i=(\psi_i+\eta_i)/2$.
We will
consistently use the notation 
$\bar\varphi:=1-\varphi$
for all expressions $\varphi$;
thus, in particular $\bar p=1-p$
and $\overline{1-p}=p$.
For $p\in(0,1)$, we write $\Ber(p)$
to denote the Bernoulli measure on $\set{0,1}$;
that is, $\Ber(p)(0)=\bar p$,
$\Ber(p)(1)=p$.
For $n\in\N$ and $p=(p_1,\ldots,p_n)\in(0,1)^n$, 
$\Ber(p)$ denotes
the product of $n$ Bernoulli distributions with parameters $p_i$:
\beq
\Ber(p) &:=& \Ber(p_1) \otimes \Ber(p_2)\otimes \ldots \otimes \Ber(p_n).
\eeq
Thus, $\Ber(p)$ is a probability measure on $\set{0,1}^n$,
with
\beq
\Ber(p)(x) 
= \prod_{i=1}^n p_i^{x_i}(1-p_i)^{1-x_i}
= \prod_{i=1}^n p_i^{x_i}\bar p_i^{\bar x_i},
\qquad
x\in\set{0,1}^n,
\eeq
We write $[n]:=\set{1,\ldots,n}$
and use
standard vector norm notation
$\nrm{w}_p^p=\sum_{i\in[n]}|w_i|^p$
for $w\in\R^n$.
For probability measures $P,Q$ on a finite set $\Omega$,
their total variation distance is
\beq
\tv{P-Q}:=\frac12\nrm{P-Q}_1=\frac12\sum_{x\in\Omega}|P(x)-Q(x)|.
\eeq
We will make 
use of 
Scheff\'e's identity
\citep[Lemma 2.1]{tsybakov09}:
\beqn
\label{eq:lecam-scheffe}
\nrm{P\wedge Q}_1
\;=\;
1-\tv{P-Q}
,
\eeqn
where $u\wedge v=\min\set{u,v}$
and 
$\sqrt{PQ}$,
$P\wedge Q$
are shorthands 
for the measures
on $\Omega$ given by
$\sqrt{P(x)Q(x)}$
and
$P(x)\wedge Q(x)$
respectively. 
For $t\in\R
$, 
$\sgn(t):=
\pred{t\ge0}
-
\pred{t<0}
$.

\begin{remark}
\label{rem:ties}
The issue of optimally breaking ties 
in \eqref{eq:fopt}
is
somewhat delicate and is exhaustively addressed
in \citet[Eq. (2.7)]{KonPin2019}.
Fortunately, although 
there may be several optimal decision rules,
they all share the same
minimum probability of error,
which depends continuously
on the parameters $\theta$, $\psi$, $\eta$.
Thus, we can always infinitesimally perturb
these so as to avoid ties, and assume no ties
henceforth.
\end{remark}

\section{Background and related work}
We refer the reader to \citet{boaz-ranking2014}
and \citet{BerendK15}
for a detailed background and literature review of this problem.
\citeauthor{boaz-ranking2014} 
and
\citet{DBLP:conf/nips/ZhangCZJ14}
proposed a spectral method for
inferring the accuracy of the experts from unsupervised data only.
Follow-up works include
\citet{
DBLP:conf/aistats/JaffeNK15,
DBLP:conf/aistats/JaffeFNJK16,
DBLP:conf/icml/ShahamCDJNCK16,
DBLP:conf/aistats/TenzerDNBK22}.

For the case of {\em symmetric} experts with 
$\psi=\eta=:p$
and $\theta=1/2$,
the optimal rule $f\opt$ given in \eqref{eq:fopt}
reduces to $f\opt(X)=\sgn\paren{
\sum_{i=1}^n w_iX_i}$,
where
$w_i:=\log\frac{p_i}{\bar p_i}$.
\citet{BerendK15}
showed that
\beqn
\label{eq:err=1-tv}
\P(f\opt\neq Y) &=& 
\frac12
\nrm{\Ber(p)\wedge\Ber(\bar p)}_1
\eeqn
and, putting
$
\Phi=\sum_{i=1}^n \paren{p_i-\frac12}w_i
$,
Theorem 1 therein
states that
\beqn
\label{eq:BK-main}
\frac{3}
{4[1+\exp(2\Phi+4\sqrt\Phi)]}
\;\le\;
\P(f\opt\neq Y)
\;\le\;
\exp(-\Phi/2)
.
\eeqn
Follow-up works include
\citet{pmlr-v48-gaoa16}
and
\citet{MANINO2019356}.
In particular,
\citeauthor[(Theorem 1, Theorem 3)]{MANINO2019356}
showed\footnote{
\citeauthor{MANINO2019356}
acknowledge
the priority of 
\cite{pmlr-v48-gaoa16}
for
both bounds in \eqref{eq:manino},
improving their constant in the lower bound
from
$0.25$ to $0.36$.
} that
\beqn
\label{eq:manino}
0.36\cdot2^n\sqrt{
\prod_{i=1}^np_i\bar p_i
}
\cdot
\exp\paren{-\frac12\sqrt{\sum_{i=1}^n w_i^2}}
\le
\P(f\opt\neq Y)
\le
\frac12
\cdot2^n
\sqrt{
\prod_{i=1}^np_i\bar p_i
}
\eeqn
and further demonstrated that
\eqref{eq:manino} sharpens both estimates in
\eqref{eq:BK-main}.

\section{Main results}
We begin with an 
analog of \eqref{eq:err=1-tv}
generalized in two ways:
the experts are neither assumed to
be symmetric
(sensitivity = specificity)
nor unbiased ($\theta=1/2$)
--- and in particular, $\theta$
explicitly figures in the expressions.

\begin{theorem}
\label{thm:errp}
For 
conditionally independent
experts 
as in
(\ref{eq:psi},\ref{eq:eta})
with sensitivities $\psi
$
and specificities $\eta
$,
the decision rule $f\opt$ in \eqref{eq:fopt} satisfies
\beq
\P(f\opt\neq Y) &=& 
\nrm{
\track{\theta}
P \wedge 
\track{\bar\theta}
Q
}_1
,
\eeq
where $P = \Ber(\psi)$
and $Q = \Ber(\bar{\eta})$.
\end{theorem}

Next, we
provide 
an
upper
bound
on
$f\opt$ in terms of the 
balanced accuracy $\pi_i=(\psi_i+\eta_i)/2$:
\begin{theorem}
\label{thm:ub}
Under the conditions of
Theorem \ref{thm:errp},
\beq
\P(f\opt\neq Y) &\le&
\sqrt{
\track{\theta \bar{\theta}}
\prod_{i=1}^n 
(\psi_i+\eta_i)
(2-\psi_i-\eta_i)
}
=
2^n
\sqrt{
\track{\theta \bar{\theta}}
\prod_{i=1}^n
\pi_i\bar\pi_i
}
.
\eeq
\end{theorem}
The above
sharpens 
the upper bound
for the symmetric case $\psi=\eta=p$
in \eqref{eq:manino} for asymmetric bias
(while recovering it for $\theta = 1/2$).
An additional interesting limiting case is
where $\psi_i=\bar{\eta_i}$ for all $i\in[n]$.
In this case, 
the experts contribute nothing,
and our upper bound evaluates to $\sqrt{\theta \bar{\theta}}$.
While this gives the exact error for $\theta = \frac 12$,
it is loose for $\theta$ close to $0$ or $1$.
At the other extreme, if
at least one of the $\pi_i\in\set{0,1}$
(which can only happen if the corresponding
$\psi_i=\eta_i=\pi_i$), the bound evaluates to $0$, as it should.

Our next result is a lower bound on 
the error probability:
\begin{theorem}
\label{thm:lb}
Under the conditions of
Theorem \ref{thm:errp},
\beq
\P(f\opt\neq Y) 
&\ge&
\track{\min \{\theta,\bar{\theta}\}}
\cdot 2^n
\sqrt{
\prod_{i=1}^n
\pi_i\bar\pi_i
}
\cdot
\exp\paren{-\frac12{\sum_{i=1}^n |\gamma_i|}}
,
\eeq
where $\gamma_i=\log(\pi_i/\bar\pi_i)$.
\end{theorem}
\track{
Note that the factor
$\sqrt{\theta \bar{\theta}}$
in the upper bound cannot be sharpened
to match the 
factor
$\min \set{\theta,\bar{\theta}}$
in the lower bound.
This is demonstrated by taking
$n=1$ and $\psi_1=\eta_1=\theta\neq1/2$.
In this case, 
$ \P(f\opt\neq Y) = \theta$,
while such a putative upper bound
would evaluate to $2\theta\sqrt{\theta\bar\theta}
<\theta$.
}

In the symmetric case,
the lower bound can be sharpened:
\begin{theorem}
\label{thm:lb-symm}
Under the conditions of
Theorem \ref{thm:errp},
where
also $\psi=\eta=p$,
\beq
\P(f\opt\neq Y) 
&\ge&
\track{\min \set{\theta,\bar{\theta}} }
\cdot 2^n
\sqrt{
\prod_{i=1}^n
p_i\bar p_i
}
\cdot
\exp\paren{-\frac12{\sqrt{\sum_{i=1}^n w_i^2}}}
,
\eeq
where $w_i=\log(p_i/\bar p_i)$.
\end{theorem}
Since
$\gamma=w$
in the symmetric case and
$\nrm{w}_2\le\nrm{w}_1\le\sqrt{n}\nrm{w}_2$,
the bound in
Theorem \ref{thm:lb-symm} is indeed significantly
sharper than that in Theorem \ref{thm:lb}.
To illustrate sharpness, consider
the case where $p_i=1/2$,
for all $i\in[n]$.
In this case, the bound evaluates to $\min \set{\theta,\bar{\theta}}$,
which is the exact value of 
$\P(f\opt\neq Y)$.
\begin{remark}
This bound is sufficiently sharp to yield the 
bound
$\tv{\Ber(p)-\Ber(\bar p)}
\le
\nrm{p-\bar p}_2
$
with the optimal constant $1$,
\citep[Theorem 3]{ak-tv24}.
\end{remark}

\paragraph{Tightness and counterexamples.}
In this subsection, we take
$\theta = 1/2$.
Theorem \ref{thm:ub} is loose in the regime $n=1$, $p=\eps$
for small $\eps$;
here,
$\P(f\opt\neq Y)\sim\eps$,
while the bound is $\sim\sqrt\eps$.
One might be tempted to improve the $\nrm{\gamma}_1$
appearing in the bound of
Theorem \ref{thm:lb} to the sharper value $\nrm{\gamma}_2$.
Unfortunately, that sharper bound does not hold.
Indeed, take $n=2$,
$\psi=(1,0)$,
and
$\eta=(1-\eps,\eps)$
for $\eps\in(0,1)$.
Then
$\pi=(1-\eps/2,\eps/2)$
and
$\gamma=
(\log(2/\eps-1),\log(\eps/(2-\eps))
$.
It is straightforward to verify that
$
\nrm{\Ber(\psi)\wedge\Ber(\bar \eta)}_1
=\eps^2
$ and that
\beq
\sqrt{
\prod_{i=1}^n
\pi_i\bar\pi_i
}
\mathe^{-\frac12\nrm{\gamma}_2}
&=&
\frac\eps2\paren{1-\frac\eps2}
\exp\paren{-\frac{
{
\log(2/\eps-1)
}}
{\sqrt2}
}
\\
&=&
\frac\eps2\paren{1-\frac\eps2}
\paren{\frac{\eps}{2-\eps}}^{1/\sqrt2}
=:f(\eps).
\eeq
Since $f(\eps)/\eps^2\to\infty$
as $\eps\downarrow0$, we conclude that the
conjectural bound fails to hold.

We can also exhibit a regime in which Theorem \ref{thm:lb-symm}
is not tight. Take $n=2$ and $p=(\eps,\eps)$ for $0<\eps<1/2$.
Then
$
\nrm{\Ber(p)\wedge\Ber(\bar p)}_1
=2\eps
$,
$w=(\log(\eps/(1-\eps),\log((1-\eps)/\eps)$, and
\beq
\sqrt{
\prod_{i=1}^n
p_i\bar p_i
}
\mathe^{-\frac12\nrm{w}_2}
&=&
\eps\paren{1-\eps}
\exp\paren{-\frac{
{
\log(1/\eps-1)
}}
{\sqrt2}
}
\\
&=&
\eps\paren{1-\eps}
\paren{\frac{\eps}{1-\eps}}^{1/\sqrt2}
=:g(\eps).
\eeq
Since $g(\eps)/\eps\to0$ as $\eps\downarrow0$,
the bound is quite loose in this regime.

\paragraph{Algorithmic aspects.}
\citet{DBLP:conf/ijcai/0001GMMPV23}
showed that for general $p,q\in[0,1]^n$, it is hard to compute 
$\tv{\Ber(p)-\Ber(q)}$ exactly.
\citet{FengApproxTV23} gave an efficient randomized algorithm for obtaining a $1\pm\eps$
multiplicative approximation with confidence $\delta$, in time $O(\frac{n^2}{\eps^2}\log\frac1\delta)$;
this was later derandomized by
\citet{Feng24Deterministically}.
Since
our results approximate 
$\nrm{P\wedge Q}_1
=
1-\tv{P-Q}
$, they are not directly comparable. Note also that
our bounds in Theorems \ref{thm:ub}, \ref{thm:lb}
are stated in terms of $p-q$ in simple, analytically tractable closed formulas. Still, as discussed above, certain gaps between the upper and lower bounds persist, and one is led to wonder to what extent these are due to computational hardness obstructions.

\section{Proofs}
We maintain our convention $\bar\varphi=1-\varphi$
for all expressions $\varphi$.
\subsection{Proof of Theorem \ref{thm:errp}}

\label{sec:pf-errp}
\noindent
The Neyman-Pearson lemma \citep[Theorem 11.7.1]{MR2239987}
implies that $f\opt$ must satisfy
\beqn
\label{eq:NP}
\P(f\opt(X)=Y|X=x)
&\ge&
\P(f\opt(X)\neq Y|X=x),
\qquad
x\in\set{0,1}^n.
\eeqn
By the Bayes formula,
an equivalent condition
is that
$f\opt(x)=1$
if and only if
\beqn
\label{eq:NP-prod}
\theta
\prod_{i\in A}\psi_i
\prod_{i\in B}\bar\psi_i
&\ge&
\bar \theta
\prod_{i\in A}
\bar\eta_i
\prod_{i\in B}\eta_i,
\eeqn
where $A,B\subseteq[n]$
are given by $A=\set{i\in[n]:x_i=1}$
and $B=\set{i\in[n]:x_i=0}$.
 Taking logarithms, \eqref{eq:NP-prod} is equivalent to
 \beqn
 \label{eq:NP-log}
 \log\frac{\theta}{\bar \theta}
 +
 \sum_{i=1}^n x_i\log\frac{\psi_i}{\bar\eta_i}
 +
 \sum_{i=1}^n \bar x_i\log\frac{\bar\psi_i}{\eta_i}
 \ge 0;
 \eeqn
 this is easily seen to be
 equivalent to 
 \eqref{eq:fopt}.
Now,
\beq
\P(f\opt(X)\neq Y)
=
\theta
\P(f\opt(X)\neq Y|Y=1)
+
\bar \theta
\P(f\opt(X)\neq Y|Y=0)
.
\eeq
Conditional on $Y=1$,
define the random variables
$Z_i=\pred{X_i=Y}$
and note that the
$(Z_1,\ldots,Z_n)$
are jointly distributed according
to $P=\Ber(\psi)$.
Putting $Q=\Ber(\bar\eta)$,
\eqref{eq:NP-prod}
implies 
that 
when $Y=1$,
$f\opt$
makes a mistake on $x\in\set{0,1}^n$
precisely\footnote{
As per Remark~\ref{rem:ties},
there is no loss of generality in assuming
no ties.
} when
$\theta P(x)<\bar{\theta} Q(x)$, whence
\beqn
\label{eq:mu1}
\P(f\opt(X)\neq Y|Y=1)
&=&
\sum_{x\in\set{0,1}^n}
P(x)
\pred{\theta P(x)<\bar{\theta} Q(x)}.
\eeqn
A similar analysis shows that
\beqn
\label{eq:nu0}
\P(f\opt(X)\neq Y|Y=0)
=
\sum_{x\in\set{0,1}^n}
Q(x)
\pred{\theta P(x) \ge \bar{\theta} Q(x)}.
\eeqn
Since
$u\pred{u<v}
+
v\pred{v\le u}
=
u\wedge v
$,
we have
\beq
\P(f\opt(X)\neq Y)
&=&
\sum_{x\in\set{0,1}^n}
\theta P(x)\wedge \bar{\theta} Q(x)
=
\nrm{\theta P \wedge \bar{\theta} Q }_1.
\eeq
which
finishes the proof.
\qed

\subsection{Proof of Theorem \ref{thm:ub}}
The following result 
may be of independent interest. Only the upper bound is used in this paper.
\begin{lemma}
\label{lem:ub-lb-pq}
For $p,q\in(0,1)^n$,
let $P,Q$ be two probability measures on
$\set{0,1}^n$ given by
$P=\Ber(p)$
and
$Q=\Ber(q)$.
Then
\beq
\sqrt{
\prod_{i=1}^n 
\frac{1-(p_i-q_i)^2}
2
}
\;\le\;
\sum_{x\in\set{0,1}^n}
\sqrt{P(x) Q(x)}
\;\le\;
\sqrt{
\prod_{i=1}^n 
[1-(p_i-q_i)^2]
}
.
\eeq
\end{lemma}
\begin{proof}
We 
prove both inequalities
by induction on $n$, starting with the second.
The base case, $n=1$, amounts to showing that
\beqn
\label{eq:base}
\sqrt{st}+\sqrt{(1-s)(1-t)}
&\le&
\sqrt{
1-(s-t)^2
},
\qquad
s,t\in(0,1).
\eeqn
Squaring both sides, \eqref{eq:base} is equivalent to
\beq
1-s-t+2st
+2 \sqrt{st(1-s) (1-t)}
\le
1-(s-t)^2
,
\eeq
which, after canceling like terms, simplifies to
\beqn
\label{eq:base-sim}
2 \sqrt{st(1-s) (1-t)}
\le
s-s^2+t-t^2.
\eeqn
Denoting the right-hand-side of \eqref{eq:base-sim} by $R$
and the left-hand-side by $L$, we compute
\beq
R^2-L^2 &=& (s-s^2-t+t^2)^2 \ge 0,
\eeq
which proves \eqref{eq:base}.
Now we assume that the claim holds for some $n=k$
and consider the case $n=k+1$:
\beq
\sum_{x\in\set{0,1}^k,y\in\set{0,1}}
\sqrt{P(x,y) Q(x,y)}
&=&
\sum_{x\in\set{0,1}^k}
\sqrt{P(x,0) Q(x,0)}
+
\sum_{x\in\set{0,1}^k}
\sqrt{P(x,1) Q(x,1)}
\\
&=&
\sum_{x\in\set{0,1}^k}
\sqrt{P(x) Q(x)
(1-p_{k+1})(1-q_{k+1})
}
\\
&+&
\sum_{x\in\set{0,1}^k}
\sqrt{P(x) Q(x)
p_{k+1}q_{k+1}
}
.
\eeq
Now apply the inductive hypothesis to each term:
\beq
\sum_{x\in\set{0,1}^k}
\sqrt{P(x) Q(x)
p_{k+1}q_{k+1}
}
&=&
\sqrt{
p_{k+1}q_{k+1}
}
\sum_{x\in\set{0,1}^k}
\sqrt{P(x) Q(x)
}
\\
&\le&
\sqrt{
p_{k+1}q_{k+1}
\prod_{i=1}^k 
[1-(p_i-q_i)^2]
}
\eeq
(the analogous bound holds for the other term).
Putting $s=
p_{k+1}$, 
$t=q_{k+1}$,
and $K=
\prod_{i=1}^k 
[1-(p_i-q_i)^2]
$,
we obtain
\beq
\sum_{x\in\set{0,1}^{k+1}}
\sqrt{P(x) Q(x)}
&\le&
\sqrt{st K}
+
\sqrt{(1-s)(1-t) K}
\\&\le&
\sqrt{
(1+s-t)(1+t-s)K
}
=
\sqrt{
\prod_{i=1}^{k+1} 
[1-(p_i-q_i)^2]
}
,
\eeq
where \eqref{eq:base} was invoked in the second inequality.
This proves the 
upper bound on
$
\sum
\sqrt{P(x) Q(x)}
$.

The lower bound proceeds in an entirely analogous fashion,
only with
\beqn
\label{eq:base'}
\sqrt{st}+\sqrt{(1-s)(1-t)}
&\ge&
\sqrt{
\frac{
1-(s-t)^2
}2
},
\qquad
s,t\in(0,1)
\eeqn
as the base case instead of \eqref{eq:base}. To prove
\eqref{eq:base'}, recall that $\sqrt{u+v}\le\sqrt{u}+\sqrt{v}$
for $u,v\ge0$ to obtain the stronger inequality
\beq
\sqrt{st+(1-s)(1-t)}
&\ge&
\sqrt{
\frac{
1-(s-t)^2
}2
}.
\eeq
Squaring both sides and collecting terms yields
the equivalent (and obviously true)
$
 (1 - s - t)^2\ge0
$. From here, the induction proceeds exactly
as in the upper bound: 
we put
$s=
p_{k+1}$, 
$t=q_{k+1}$,
and $K=
\prod_{i=1}^k 
[1-(p_i-q_i)^2]
$,
and repeat the steps therein with the inequality appropriately reversed.

\end{proof}

\begin{proof}[of Theorem \ref{thm:ub}]
By Theorem \ref{thm:errp},
$\P(f\opt\neq Y)
=
\nrm{\theta P \wedge \bar{\theta} Q}_1
$.
Now
since $a\wedge b\le \sqrt{ab}$
for $a,b
\ge0
$, we have
\beqn
\label{eq:min<sqrt}
\sum_{x\in\set{0,1}^n}
P'(x)\wedge Q'(x)
&\le&
\sum_{x\in\set{0,1}^n}
\sqrt{P'(x) Q'(x)}
\eeqn
for all positive measures
$P',Q'$ on $\set{0,1}^n$.
Setting $P' = \theta P$ and $Q' = \bar{\theta} Q$,
we now have that 
\beqn
\sum_{x\in\set{0,1}^n}
\theta P(x)\wedge \bar{\theta} Q(x)
&\le&
\sqrt{\theta \bar{\theta}}
\sum_{x\in\set{0,1}^n}
\sqrt{P(x) Q(x)}.
\eeqn
Applying the upper bound in
Lemma~\ref{lem:ub-lb-pq} with $p_i=\psi_i$
and $q_i=1-\eta_i$
and noting that
$
1-(\psi_i-\bar\eta_i)^2
=
4\pi_i\bar\pi_i
$
completes the proof.
\end{proof}

\subsection{Proof of Theorem \ref{thm:lb}}
\noindent
As
$\P(f\opt\neq Y)
=
\nrm{\theta P \wedge \bar{\theta} Q}_1
$,
we start by writing
\begin{align*}
\nrm{\theta P \wedge \bar{\theta} Q}_1
\ge
\min \set{\theta,\bar{\theta}} 
\nrm{P \wedge Q}_1
=
\min \set{\theta,\bar{\theta}} 
\nrm{\Ber(\psi)\wedge\Ber(\bar\eta)}_1,
\end{align*}
where we used the pointwise inequality
\begin{align*}   
\min \set{\lambda u , \bar{\lambda}v} 
\ge 
\min \set{\lambda,\bar{\lambda}} 
\min \set{u,v}.
\end{align*}

\noindent
Next, we invoke Lemma \ref{lem:min-tens} with 
$P=\Ber(\psi)$
and
$Q=\Ber(\bar\eta)$ to obtain
\beq
\nrm{\Ber(\psi)\wedge\Ber(\bar\eta)}_1
&\ge&
\prod_{i=1}^n
\nrm{\Ber(\psi_i)\wedge\Ber(\bar\eta_i)}_1
\\
&=&
\prod_{i=1}^n
\sqprn{
\psi_i\wedge\bar\eta_i
+
\bar\psi_i\wedge\eta_i
}
\\
&\ge&
\prod_{i=1}^n
2(\pi_i\wedge\bar\pi_i),
\eeq
where the last inequality is due to \eqref{eq:stumin}.
By \eqref{eq:uvmin},
\beq
\pi_i\wedge\bar\pi_i
&=&
\sqrt{\pi_i\bar\pi_i}\exp\paren{-\frac12
\abs{\log\frac{\pi_i}{\bar\pi_i}}
},
\eeq
whence
\beq
\prod_{i=1}^n
2(\pi_i\wedge\bar\pi_i)
&=&
2^n
\sqrt{
\prod_{i=1}^n
\pi_i\bar\pi_i
}
\cdot
\exp\paren{-\frac12{\sum_{i=1}^n 
\abs{\log\frac{\pi_i}{\bar\pi_i}}
}}
.
\eeq
This proves the claim.
\qed

\subsection{Proof of Theorem \ref{thm:lb-symm}}
\noindent
\track{
Repeating the argument from Theorem \ref{thm:lb},
we have 
\begin{align}
\P(f\opt\neq Y)
\ge 
\min \set{\theta,\bar{\theta}} 
\nrm{\Ber(\psi)\wedge\Ber(\bar\eta)}_1.  
\end{align}
}
Setting $\psi_i = \eta_i = p_i$
we 
invoke \eqref{eq:uvmin} to obtain
\beq
\nrm{\Ber(p)\wedge\Ber(\bar p)}_1
&=&
\sum_{x\in\set{0,1}^n}
P(x)\wedge Q(x)
\\
&=&
\sum_{x\in\set{0,1}^n}
\sqrt{P(x) Q(x)}
\exp\paren{-\frac12\abs{\log\frac{P(x)}{Q(x)}}}
\\
&=&
\sqrt{
\prod_{i=1}^n
p_i\bar p_i
}
\sum_{x\in\set{0,1}^n}
\exp\paren{-\frac12\abs{\log\frac{P(x)}{Q(x)}}}
\\
&=&
2^n
\sqrt{
\prod_{i=1}^n
p_i\bar p_i
}
\E_{Z\sim\mathrm{Uniform}\set{0,1}^n}
\exp\paren{-\frac12\abs{\log\frac{P(Z)}{Q(Z)}}}
\\
&\ge&
2^n
\sqrt{
\prod_{i=1}^n
p_i\bar p_i
}
\exp\paren{-\frac12\E_{Z
}\abs{\log\frac{P(Z)}{Q(Z)}}},
\eeq
where Jensen's inequality was used in the last step.
Since $P,Q$ are product measures, we have
$P(Z)=\prod_{i=1}^n P_i(Z_i)$
and
$Q(Z)=\prod_{i=1}^n Q_i(Z_i)
=\prod_{i=1}^n \bar P_i(Z_i)
$, whence
\beq
\E
\abs{\log\frac{P(Z)}{Q(Z)}}
&\le&
\sqrt{
\E
\sqprn{
\paren{\log\frac{P(Z)}{Q(Z)}}^2
}}
\\
&=&
\sqrt{
\E
\sqprn{
\paren{
\sum_{i=1}^n
\log\frac{P_i(Z_i)}{
\bar P
_i(Z_i)}}^2
}}
\\
&=&
\sqrt{
\E
\sqprn{
\sum_{i,j\in[n]}
\log\frac{P_i(Z_i)}{
\bar P
_i(Z_i)}
\log\frac{P_j(Z_j)}{
\bar P
_j(Z_j)}
}}.
\eeq
Since $\E
\log\frac{P_i(Z_i)}{
\bar P
_i(Z_i)}=0$
and the $Z_i$ are independent,
only the diagonal terms survive:
\beq
\E
\sqprn{
\sum_{i,j\in[n]}
\log\frac{P_i(Z_i)}{
\bar P
_i(Z_i)}
\log\frac{P_j(Z_j)}{
\bar P
_j(Z_j)}
}
&=&
\sum_{i=1}^n
\E\abs{
\log\frac{P_i(Z_i)}{
\bar P
_i(Z_i)}}^2
\\&=&
\sum_{i=1}^n\paren{\log\frac{p_i}{\bar p_i}}^2.
\eeq
The proof is complete.
\qed

\subsection{Auxiliary Lemmata}

The following 
identity and inequality
are elementary:
\beqn
u\wedge v
&=&
\sqrt{uv}\exp\paren{-\frac12\abs{\log\frac{u}{v}}},
\qquad
u,v>0,
\label{eq:uvmin}
\\
\label{eq:stumin}
s\wedge\bar t+t\wedge\bar s
&\ge& 2(u\wedge\bar u),
\qquad
s,t\in[0,1],
u=
(s+t)/2
.
\eeqn

\begin{lemma}
\label{lem:min-tens}
For all probability measures
$P,Q,P',Q'$,
we have
\beq
\nrm{
P\otimes Q 
\wedge
P'\otimes Q'
}_1
&\ge&
\nrm{
P
\wedge
P'
}_1
\cdot
\nrm{
Q 
\wedge
Q'
}_1
.
\eeq
\end{lemma}
\begin{proof}
It is a classic fact (see, e.g., \citet[Lemma 2.2]{kontorovich12})
that
\beq
\tv{
P\otimes Q 
-
P'\otimes Q'
}
&\le&
\tv{P-P'}
+
\tv{Q-Q'}
-
\tv{P-P'}
\cdot
\tv{Q-Q'}
.
\eeq
The
claim follows
by
Scheff\'e's identity \eqref{eq:lecam-scheffe}.
\end{proof}

\section*{Acknowledgments}
We thank 
Daniel Berend, 
Lior Daniel,
Ariel Jaffe,
Douglas Hubbard, 
Sudeep Kamath,
Mark Kozlenko, 
Kuldeep Meel,
Boaz Nadler, and Rotem Zur for enlightening discussions.

\bibliography{refs}

\end{document}